\documentclass[12pt,twoside]{amsart}

\usepackage{pdfsync}

\usepackage{amssymb,amsmath,amsfonts,latexsym}
\usepackage{color}
\usepackage{geometry}

\usepackage{mathrsfs}

\newcommand{\f}{\frac}
\newcommand{\om}{\Omega}
\newcommand{\omn}{\omega_n}

\newtheorem{thm}{Theorem}

\newtheorem{prop}{Proposition}

\begin{document}

\title[Equivalent and attained version of Hardy's inequality in $\mathbb{R}^n$]{Equivalent and attained version\\ of Hardy's inequality in $\mathbb{R}^n$}
\author{D.~Cassani \and B.~Ruf \and C.~Tarsi }

\address[Daniele Cassani]{\newline\indent Dip. di Scienza e Alta Tecnologia
\newline\indent
Universit\`{a} degli Studi dell'Insubria
\newline\indent and
\newline\indent RISM--Riemann International School of Mathematics
\newline\indent via G.B.~Vico 46, 21100 - Varese, Italy \newline\indent\texttt{daniele.cassani@uninsubria.it}
}

\address[Bernhard Ruf and Cristina Tarsi]{\newline \indent Dip. di Matematica \newline\indent Universit\`a degli Studi di Milano\newline\indent via C.~Saldini 50, 20133 - Milano, Italy\newline\indent \texttt{bernhard.ruf@unimi.it}\newline\indent\texttt{cristina.tarsi@unimi.it}
}

\date{\today}
\begin{abstract}
We investigate connections between Hardy's inequality in the whole space $\mathbb{R}^n$ and embedding inequalities for
Sobolev-Lorentz spaces. In particular, we complete previous results due to \cite[Alvino]{Alv1} and \cite[Talenti]{Ta} by establishing optimal embedding inequalities for the Sobolev-Lorentz quasinorm $\|\nabla\,\cdot\,\|_{p,q}$ also in the range $p < q<\infty$, which remained essentially open since \cite{Alv1}.
\par \smallskip \noindent
Attainability of the best embedding constants is also studied, as well as the
limiting case when $q=\infty$. Here, we surprisingly discover that the Hardy inequality is equivalent to the corresponding Sobolev-Marcinkiewicz embedding inequality. Moreover, the latter turns out to be attained by the so-called ``ghost'' extremal functions of \cite[Brezis-Vazquez]{BV}, in striking contrast with the Hardy inequality, which is never attained. In this sense, our functional approach seems to be more natural than the classical Sobolev setting, answering a question raised in \cite{BV}.
\end{abstract}
\maketitle

\section{Introduction}
\noindent The classical Hardy inequality for smooth compactly supported functions in $\Omega\subseteq\mathbb{R}^n$ and for $1<p<n$, reads as follows
\begin{equation}\label{Hi}\tag{$H_p$}
\left(\f{n-p}{p}\right)^p\int_{\Omega}\f{|u|^p}{|x|^p}dx\leq
\int_{\Omega} |\nabla u|^p dx
\end{equation}
where the constant in the left hand side of \eqref{Hi} is sharp for any sufficiently smooth domain containing the origin. Actually, Hardy proved in 1925 the one dimensional version of \eqref{Hi},
see \cite{KMP,D} for a historical insight into the subject. The original result has been extended and generalized by many authors in several
directions which break through different aspects of Analysis, Geometry and PDE, among which we mention
\cite{Ma,OK,M,BM,BMS,BV,CF,FT,GM,MMP}.

\par \smallskip
While much  progress has been achieved in understanding \eqref{Hi} and its  generalizations, a basic question raised by Brezis and Vazquez in \cite{BV} on the attainability of the best constant in \eqref{Hi} has not been given a full answer yet. Indeed, in \cite{BV,Ma}  it was found that additional lower order terms are admissible on the left hand side of \eqref{Hi},
as long as $\Omega$ stays bounded, and an extensive literature has been devoted to searching for such {\it remainder terms}
in Hardy and Hardy-type inequalities, see \cite{FT,CF} and references therein. This phenomenon yields an
obstruction to the attainability of the best constant in \eqref{Hi}, provided the domain $\Omega$ contains the origin.

When $\Omega=\mathbb{R}^n$ the existence of a suitable class of remainders has been recently established in \cite{CF, SanoTakahashi}, see also \cite{GhoussoubMoradifam}. As mentioned, the presence of remainders prevents that the Hardy inequality is attained, and we refer also to the recent papers \cite{pincho1,pincho2} for a deeper understanding of this phenomenon. In particular, the Euler-Lagrange equation corresponding to the equality case in Hardy's inequality has no solution in the Sobolev space $\mathcal{D}^{1,p}(\mathbb{R}^n)$, defined as the completion of $C_c^\infty(\mathbb R^N)$ with respect to the norm $\|\nabla \, \cdot \, \|_p$,  however it is explicitly solved by a class of functions which do not belong to  this space. The lack of a proper function space setting was  pointed out in \cite{BV} and has inspired  our work since the very beginning.
\par \smallskip
Another interesting aspect of \eqref{Hi} is its equivalence to the optimal Sobolev embedding for the space
$\mathcal{D}^{1,p}(\mathbb{R}^n)$ in the context of Lorentz spaces, namely
\begin{equation}\label{A_in}\tag{$A_{p,p}$}
 \|u\|_{p^\ast, p}\leq S_{n,p}\|\nabla u\|_p
\end{equation}
which was obtained by Alvino in \cite{Alv1}, see also Peetre \cite{P}.
The constant
$$
S_{n,p}=\f{p}{n-p}\f{\left[\Gamma(1+n/2)\right]^{\f{1}{n}}}{\sqrt{\pi}}=
\f{p}{n-p}\,\omn^{-\f 1n}
$$
is  best possible and the embedding given by \eqref{A_in} is optimal in regards to the target space $L^{{p^*},p}$ which is  smallest
 among all rearrangements invariant spaces \cite{BS,EKP} ($\Gamma$ denotes the standard Euler Gamma function and
 $\omega_n$ stands for the measure of the unit sphere in $\mathbb{R}^n$).
  In this sense, \eqref{A_in} yields the optimal version of the Sobolev embedding theorem.

\par \smallskip
 The equivalence between \eqref{Hi} and \eqref{A_in} is a consequence of the P\'olya-Szeg\"o principle and
 the Hardy-Littlewood inequality by which the left hand side of \eqref{Hi} does not increase under radially
decreasing symmetrization and it is equal to the left hand side of \eqref{A_in} when $u$ is radially decreasing.
\par \smallskip
\noindent Alvino in \cite{Alv1} proved actually the following inequalities
\begin{equation}\label{A_in2}\tag{$A_{p,q}$}
\|u\|_{p^*,q}\leq \f{p}{n-p}\,\omn^{-\f 1n} \|\nabla u\|_{p,q},
\quad 1\leq p<n
\end{equation}
related to the Sobolev-Lorentz embedding
\begin{equation}\label{emb}
\mathcal{D}_H^1L^{p,q}(\mathbb R^n) \hookrightarrow
L^{p^*,q}(\mathbb R^n)
\end{equation}
with the restriction
$$
1\leq q \leq p\, ,
$$
see also \cite{CRT}. The homogeneous Sobolev-Lorentz space
$\mathcal{D}_H^1L^{p,q}(\mathbb R^n)$ is obtained as the closure
of smooth compactly supported functions with respect to the
Lorentz quasi-norm $\|\nabla\,\cdot\,\|_{p,q}$. Note that the
validity of the embedding \eqref{emb}  for $1\leq q\leq +\infty$
is well known in the literature of interpolation theory:
more direct and short proofs can be found in  \cite{Ta, ALT}.

\par \smallskip \noindent
 Let us point out that the embedding
constant in \eqref{A_in2} is sharp, and it does not depend on the
second Lorentz index $q$. Moreover, up to a normalizing factor, it
turns out to be the Hardy constant.

 \vspace*{0.3cm}

\subsection*{Main results}
$ $
\par \smallskip
\noindent Our first goal is to extend the validity of the embedding inequality \eqref{A_in} to the values $p < q\leq \infty$, still preserving the optimal constant, thus completing the results of Alvino \cite{Alv1} and Talenti \cite{Ta} to the whole range $1 \le q \le \infty$.
\par \smallskip
In the case $q=\infty$ the functional setting is somewhat delicate, as no Meyer-Serrin type result
holds for the corresponding homogeneous spaces. Thus, let us define for $1\leq p<\infty$ and $1\leq q \leq \infty$
the space
$$
\mathcal{D}_W^1L^{p,q}(\mathbb R^n):=\left\{u\in L^{p^*,q}(\mathbb R^n): |\nabla
u| \in L^{p,q}(\mathbb R^n)\right\}\ .
$$
Then it turns out that for $q<\infty$
\begin{equation*}
\mathcal{D}_H^1L^{p,q}(\mathbb R^n)=
\mathcal{D}_W^1L^{p,q}(\mathbb R^n)=: \mathcal{D}^1L^{p,q}(\mathbb
R^n)
\end{equation*}
whereas for $q=\infty$ one has
$$\mathcal{D}_H^1L^{p,\infty}(\mathbb R^n)\subsetneq
\mathcal{D}_W^1L^{p,\infty}(\mathbb R^n),
$$
see \cite{C} and Section \ref{preliminaries} for more details.
\par \bigskip
\begin{thm}\label{thm generAlv}
Let $1\leq p<n$, $p< q\leq \infty$. Then the   following
inequality holds for any $u\in \mathcal{D}_W^1L^{p,q}(\mathbb R^n)$
  \begin{equation}\label{SobLorineq-R^n}\tag{$A_{p,q}$}
    \|u\|_{p^{\ast},q}\leq \f{p}{n-p}\,\omn^{-\f 1n} \|\nabla
    u\|_{p,q}
  \end{equation}
where the constant $\f{p}{n-p}\,\omn^{-\f 1n}$ is sharp.
\end{thm}
\par \bigskip
\noindent Then, surprisingly, we establish the equivalence between $(A_{p,\infty})$ and $(H_p)$.
\par \medskip
\begin{thm}\label{thm equivalence}
Let $1\leq p<n$. Then, Hardy's inequality
\begin{equation}\label{hardyequiv}\tag{$H_{p}$}
 \Big(\f{n-p}{p}\Big)^p\int_{\mathbb R
^n}\f{|u|^p}{|x|^p}dx\leq \int_{\mathbb R^n} |\nabla u|^p dx
\end{equation}
holds for any $u\in\mathcal{D}^{1,p}(\mathbb R^n)$ if and only if the
Sobolev-Marcinkiewicz embedding inequality
\begin{equation}\label{weakA}\tag{$A_{p,\infty}$}
    \|v\|_{p^{\ast},\infty}\leq
\f{p}{n-p}\,\omn^{-\f 1n} \|\nabla
    v\|_{p,\infty}
\end{equation}
holds for any $v\in \mathcal{D}_W^1L^{p,\infty}(\mathbb R^n)$.
\end{thm}
\par \bigskip
Finally we study the attainability of $(A_{p,q})$. In particular, in the limiting case $q=\infty$, ($A_{p,\infty}$) turns out to be attained, in striking contrast to Hardy's inequality, regardless of their equivalence as established in Theorem \ref{thm equivalence}.
\begin{thm}\label{thm attainedornot}
Let $1\leq p<n$ and $1\le q\leq \infty$. Then, the sharp constant in
\eqref{SobLorineq-R^n} is attained if and only if $q=+\infty$.
Moreover, an extremal function for  $(A_{p,\infty})$ in $\mathcal
D^1_WL^{p,\infty}$ is given by
\begin{equation*}
\psi(x)= |x|^{-\frac{n-p}{p}}
\end{equation*}
\end{thm}
\par \medskip \noindent
{\bf Remark:} The extremal function in Theorem \ref{thm attainedornot} is exactly the ``ghost'' extremal function of \cite[Brezis-Vazquez]{BV}.
\par \bigskip
\subsection*{Overview}
$ $
\par \smallskip
\noindent
In Section \ref{preliminaries} we recall for convenience some well known facts and prove a few preliminary results. Then, in Section \ref{H_implies_A} we prove Theorem \ref{thm generAlv} by showing that  \eqref{A_in2} for $p<q\leq \infty$ can be obtained as a consequence of \eqref{A_in}, which is actually equivalent to Hardy's inequality. The proof relies on suitable scaling properties whereas the sharpness of the embeddng constants is proved by inspection. As a byproduct of Theorem \ref{H_implies_A}, the sharp Marcinkiewicz type inequality $(A_{p,\infty})$ in
$\mathcal{D}^1_WL^{p,\infty}$ is a consequence of $(A_{p,p})$,
that is of the Hardy inequality \eqref{Hi}. In Section \ref{weaktostrong}, we surprisingly prove also the converse, namely that the validity of $(A_{p,\infty})$ in
$\mathcal{D}^1_WL^{p,\infty}$ implies Hardy's inequality
\eqref{Hi} in $\mathcal{D}^{1,p}$. In Section \ref{att} we prove that the best constant in $(A_{p,q})$ is never attained as long as $q<\infty$, and then attained at the endpoint of the Lorentz scale for $q=\infty$. This is in striking contrast with Hardy's inequality which is never attained though being equivalent to $(A_{p,\infty})$. In this sense our functional framework, namely the Sobolev-Marcinkiewicz space
$\mathcal{D}_W^1L^{p,\infty}$, seems to be qualified as more natural than the classical $\mathcal{D}^{1,p}$ setting in the tradition of Hardy type inequalities. This phenomenon throws light on the importance of considering the couple (inequality, functional setting), as the whole information retained can be differently shared between the two components through equivalent versions.
Finally, in the Appendix we recall and adapt to our situation a by now standard technique to reduce the embedding problems to the radial case, as initially developed by \cite[Alvino-Lions-Trombetti]{ALT}.
\par \medskip
\section{Preliminaries}\label{preliminaries}

\noindent For convenience of the reader, let us briefly recall some basic facts on Lorentz spaces \cite{L} which will be widely used throughout the paper. \\
\noindent For a
measurable function $u: \Omega \to \mathbb R^+$, let $u^*$ denote
its decreasing rearrangement which is defined as the
distribution function of the distribution function $\mu_{u}$ of
$u$, namely
$$
u^*(s)=|\{t\in[0,+\infty)\,:\,\mu_{u}(t)>s\}|=\sup\{t>0\,:\,
\mu_{u}(t)>s\}, \quad s\in [0, |\Omega|]
$$
whereas the spherically symmetric rearrangement
$u^{\#}(x)$ of $u$ can be defined as
$$
u^{\#}(x)=u^{*}(\omega_n |x|^n), \quad x\in \Omega^{\#}
$$
where $\Omega^{\#}\subset \mathbb{R}^n$ is the open ball with
center in the origin which satisfies $|\Omega^{\#}|=|\Omega|$ and
$\omega_n$ is the area of the unit sphere of $\mathbb R^n$.
Clearly, $u^*$ is a nonnegative, non-increasing and
right-continuous function on $[0,\infty)$. Moreover, the
(nonlinear) rearrangement operator enjoys the following
properties:
\begin{enumerate}
\item[i)] Positively homogeneous: $(\lambda
u)^*=|\lambda|u^*,\quad\lambda\in\mathbb R$;
\vspace{.2cm}
\item[ii)]
Sub-additive: $(u+v)^*(t+s)\leq u^*(t)+v^*(s),\quad t,s\geq 0$;
\vspace{.2cm}
\item[iii)] Monotone: $0\leq u(x)\leq v(x) \hbox{ a.e.~in } \om
\Rightarrow u^{\ast}(t)\leq v^{\ast}(t),\quad t\in (0,|\om|)$;
\vspace{.2cm}
\item[iv)] $u$ and $u^*$ are equidistributed and in particular (Cavalieri's principle)
$$\int_\om A(|u(x)|)\,dx=\int_0^{|\om|}A(u^*(s))\,ds$$ for any
continuous funtion $A:[0,\infty]\rightarrow [0,\infty]$,
nondecreasing and such that $A(0)=0$;
\vspace{.2cm}
\item[v)] The following
inequality holds (Hardy-Littlewood): $$\int_\om
u(x)v(x)\,dx\leq\int_0^{|\om|}u^*(s)v^*(s)\,ds$$ provided the
integrals involved are well defined.
\vspace{.2cm}
\item[vi)] The map $u\mapsto u^*$ preserves
Lipschitz regularity, namely $$^*: Lip(\om)\longrightarrow
Lip(0,|\om|)$$.
\end{enumerate}

\noindent Then, the Lorentz space $L^{p,q}(\Omega)$ is a rearrangement invariant Banach space \cite{BS} which can be defined as follows
$$
L^{p,q}(\Omega) := \bigg\{ u : \Omega \to \mathbb R \
\hbox{measurable} \ \big| \ \|u\|_{p,q} :=
\Big(\int_0^{\infty}\left(u^{\ast}(t)t^{1/{p}}\right)^q\frac{dt}{t}\Big)^{\frac{1}{q}}
< \infty \bigg\}
$$
where the quantity $\|u\|_{p,q}$ is a quasi-norm which admits an equivalent norm. One clearly has $L^{p,p} = L^p$
and furthermore, with respect to the second index, Lorentz spaces
satisfy the following inclusions (Lorentz scale)
\begin{equation}\label{lor}
 L^{p,q_1} \subsetneq L^{p,q_2} \ , \ \hbox{ if } \ 1 \le q_1 < q_2  \le \infty \
\end{equation}
For $q=\infty$ we obtain the so-called Marcinkiewicz or weak-$L^p$
space, which is defined as follows
\begin{equation}\label{wqn}
 \|u\|_{p,\infty}:=\sup_{t>0} t^{\f{1}{p}}u^*(t)\
\end{equation}
\noindent Notice that in particular one has $L^{p^*,p}\subsetneq L^{p^*,p^*}=L^{p^*}$.

\vspace*{0.2cm}

\noindent Sobolev-Lorentz spaces generalize classical
Sobolev spaces. First order Sobolev-Lorentz spaces can be defined either as the closure of smooth
compactly supported functions $u$ with respect to the norm
$\|\nabla u\|_{p,q}+\|u\|_{p,q}$, or as the set of functions in
$L^{p,q}(\Omega)$ whose distributional gradient also belongs to
$L^{p,q}$. We refer for the general theory on Sobolev-Lorentz spaces to \cite{C} and references therein, and to \cite{CP1} for more general
Sobolev spaces realized on rearrangement invariant
Banach spaces.

\noindent Here we focus on {\emph{homogeneous Sobolev-Lorentz}} spaces defined for $1\leq p<+\infty$ and $1\leq q\leq +\infty$ by
\begin{equation*}
  \mathcal D_H^1L^{p,q}=cl\big\{u\in \mathcal C^{\infty}_c(\mathbb R^n): \|\nabla
u\|_{p,q}<\infty\big\}\
\end{equation*}
Since $\mathcal
D_H^1L^{p,q}\hookrightarrow L^{p^*,q}$, as a consequence of
\cite{Alv1, Ta}, one may also consider the alternative definition
given by
\begin{equation*}
  \mathcal D_W^1L^{p,q}=\left\{u\in L^{p^*,q}(\mathbb R^n): \|\nabla
u\|_{p,q}<\infty\right\}\ .
\end{equation*}
It turns out that the two spaces coincide as long as $q<\infty$
\cite[Section 4]{C}:
$$
\mathcal D_W^1L^{p,q}= \mathcal D_H^1L^{p,q}:=\mathcal D^1L^{p,q}
\quad \hbox{ if } 1\leq p,q<\infty
$$
whereas  in the {\emph{limiting case }} $q=\infty$, we have
$$
\mathcal D_H^1L^{p,\infty} \subsetneq \mathcal D_W^1L^{p,\infty}
$$
A function  belonging to $\mathcal D_W^1L^{p,\infty}\setminus
\mathcal D_H^1L^{p,\infty}$ is given by
$u(x)=|x|^{-\frac{n-p}{p}}$ (\cite[Prop.~4.7]{C}).

\noindent Sobolev-Lorentz spaces enjoy invariance properties by scaling. As a consequence, inequalities \eqref{SobLorineq-R^n} and in particular the Hardy inequality \eqref{Hi} are invariant under the action of the group of dilations, as established in the following
\begin{prop}\label{dilation_inv}
 Let $\lambda>0$, $1\leq p<n$, $1\leq q\leq \infty$ and $u_\lambda(x):=u(\lambda x)$. Then, the following quotients
 \begin{eqnarray*}
  \displaystyle \frac{\|\nabla
u_{\lambda}\|_{p,q}}{\|u_{\lambda}\|_{p^*,q}}\ \ , \quad &&u\in \mathcal{D}^1L^{p,q}(\mathbb{R}^n) \\
&&\\
\displaystyle \frac{\displaystyle\int_{\mathbb R^n}
\frac{|u_{\lambda}|^p}{|x|^p}\,dx}{\|\nabla u_{\lambda}\|_p^p}\ \ , && u\in \mathcal D^{1,p}(\mathbb{R}^n)
\end{eqnarray*}
are constant with respect to $\lambda$.
\end{prop}
\begin{proof}
Let us first consider the case $q<\infty$. We have  $u^\sharp_\lambda(|x|)=u^\sharp(\lambda |x|)$, $u^*_\lambda(t)=u^*(\lambda^n t)$ and
  \begin{align*}
  \|u_\lambda\|_{p^*,q}^q &=\int_{0}^{+\infty}\left[u^*_\lambda (t)t^{\frac
  1{p^*}}\right]^q\frac{dt}{t}=\int_{0}^{+\infty}\left[u^* (\tau)(\lambda^{-n}\tau)^{\frac
  1{p^*}}\right]^q\frac{d \tau}{\tau}\\
  &=\lambda^{-\frac{nq}{p^*}} \|u\|_{p^*,q}^q
  \end{align*}
 From $\nabla u_\lambda(x)=\lambda \nabla u(x)$, we have
  $|\nabla u_\lambda|^\sharp(|x|)=\lambda |\nabla u|^\sharp(\lambda |x|)$ and $|\nabla u_\lambda|^*(t)=
  \lambda |\nabla u|^*(\lambda^n t)$. Hence
\begin{align*}
  \|\nabla u_\lambda\|_{p,q}^q &=\int_{0}^{+\infty}\left[|\nabla u_\lambda|^* (t)t^{\frac
  1{p}}\right]^q\frac{dt}{t}=\int_{0}^{+\infty}\left[\lambda |\nabla u|^* (\tau)(\lambda^{-n}\tau)^{\frac
  1{p}}\right]^q\frac{d \tau}{\tau}\\
&=\lambda^{(-\frac{n}{p}+1)q} \|\nabla u\|_{p,q}^q
\end{align*}
and the first claim follows as
$(-\frac{n}{p}+1)=-\frac{n}{p^*}$\ .\\
When $q=\infty$ we have,
\begin{align*}
  \|u_\lambda\|_{p^*,\infty}&=\sup_{t}t^{1/p^*}u_\lambda^*(t)=
  \lambda^{-n/p^*}\sup_{\tau}\tau^{1/p^*}u^*(\tau)\\
&= \lambda^{-n/p^*}\|u\|_{p^*, \infty}
  \end{align*}
 as well as
\begin{align*}
  \|\nabla u_\lambda\|_{p,\infty}&=\sup_{t}t^{1/p}\nabla u_\lambda^*(t)=
  \lambda^{-n/p+1}\sup_{\tau}\tau^{1/p^*}\nabla u^*(\tau)\\
&= \lambda^{-n/p+1}\|u\|_{p, \infty}
\end{align*}
and the first claim follows as above.

\noindent The second claim follows by observing that
$$
\int_{\mathbb R^n}\frac{|v_\lambda|^p}{|x|^p}\,dx= \lambda^{-n+p}
\int_{\mathbb R^n}\frac{|v|^p}{|x|^p}\,dx
$$
\end{proof}

\section{Proof of Theorem \ref{thm generAlv} }\label{H_implies_A}
\noindent Next we will prove that the sharp embedding inequality $(A_{p,p})$ implies all the
embedding inequalities $(A_{p,q})$, $p<q\leq
+\infty$, still preserving the sharp embedding constants. The proof
strongly relies on the reduction to the radial case, which is a
rather delicate issue in the case $q>p$: indeed, the argument used in
\cite{Alv1} for $q<p$, is based on a generalization of the P\'olya-Szeg\"o result, which cannot be applied here. However, following the approach of \cite{ALT}, one can
prove that for any $u\in \mathcal C^{\infty}_c(\mathbb R^n)$ there
exists $v\in \mathcal{D}^{1,\sharp}L^{p,q}(\mathbb R^n)$, namely
$v\in \mathcal{D}^{1}L^{p,q}(\mathbb R^n)$ and it is radial and
monotone decreasing, such that
$$
\|v\|_{p^*,q}\geq \|u\|_{p^*,q} \quad \hbox{ and } \quad \|\nabla
v\|_{p,q}\leq \|\nabla u\|_{p,q}
$$
This fact allows to restrict to radial decreasing functions also in
the case $p<q<+\infty$. Though the argument is standard by now, we outline the
details in the Appendix.
\par \medskip
\noindent The proof of Theorem \ref{thm generAlv} is based on scaling arguments. Let us divide the proof into two steps.
\medskip

\noindent \textbf{Step 1.} The case $p<q<\infty$.
%(\emph{NB: A simpler proof!  Choose $s=\f{q}{q-p}$ in the previous   version.})

\smallskip

 \noindent Let $ u\in \mathcal{D}^1L^{p,q}(\mathbb R^n)$ such that $u=u^\sharp$ and define
the radially decreasing function
$$
v(x):=\left[u(|x|^{\f pq})\right]^{\f
qp}=\left[u^*(|x|^{\f{np}q}\omn)\right]^{\f qp}
$$
so that
$$
v^{\sharp}(|x|)=\left[u(|x|^{\f pq})\right]^{\f qp}\ , \
v^*(t)=v^{\sharp}\big((\textstyle{\frac{t}{\omn}})^{1/n}\big)%
=\left[u\big((\textstyle{\frac{t}{\omn}})^{\f{p}{qn}}\big)\right]^{\f
qp}=\Big[u^*(t^{\f pq}\omn^{\f{q-p}{q}})\Big]^{\f qp}
$$
\noindent One has $v\in L^{p^*,p}(\mathbb{R}^n)$, indeed
\begin{eqnarray}\label{vnorm2}
  \nonumber \|v\|_{p^*,p}&=&\left\{\int_0^{\infty}\left[v^*(t)t^{1/p^*}\right]^{p}\f{dt}{t}\right\}^{1/p}\\
  \nonumber &=& \left\{\int_0^{\infty}\left[(u^*(t^{\f pq}\omn^{\f{q-p}{q}}))^{\f qp}t^{1/ p^*}\right]^{
  p}\f{dt}{t}\right\}^{1/p}\\
 \nonumber  &=&\left\{\int_0^{\infty}\left[u^*(t^{\f pq}\omn^{\f{q-p}{q}})t^{\f{n-p}{nq}}\right]^{q}\f{dt}{t}\right\}^{1/p}\\
  \nonumber &=& \left(\f qp\right)^{\f 1p}\left\{\int_0^{\infty}\left[u^*(\tau)\tau^{\f{1}{p^*}}\omn^{\f{p-q}{q p^*}}\right]^{q}%
  \f{d\tau}{\tau}\right\}^{1/p}\\
  \nonumber&=& \left(\f qp\right)^{\f 1p}\omn^{\f{p-q}{p p^*}}%
\left\{\int_0^{\infty}\left[u^*(\tau)\tau^{\f{1}{p^*}}\right]^{q}\f{d\tau}{\tau}\right\}^{1/p}\\
  &=& \left(\f qp\right)^{\f 1p}\omn^{\f{p-q}{p p^*}}\|u\|^{\f qp}_{p^*,  q}%
\end{eqnarray}
Moreover, $v\in \mathcal{D}^{1,p}(\mathbb{R}^n)$ as one has
$$
|\nabla v(x)|=\big|u(|x|^{\f pq})\big|^{\f{q-p}{p}}\big|\nabla u(|x|^{\f
pq})\big||x|^{\f{p-q}{q}}
$$
so that %
\begin{eqnarray*}
\|\nabla v\|_{p}  &=&\left\{\int_{\mathbb R^n}|\nabla
  v|^{p}dx\right\}^{1/p}\\
  &=&\left\{\int_{\mathbb R^n}\left[|u(|x|^{\f pq})|^{\f{q-p}{p}}|\nabla u(|x|^{\f
pq})||x|^{\f{p-q}{q}}\right]^{p}dx\right\}^{1/p}
\end{eqnarray*}
here the fact $q>p$ is crucial. Next apply the (generalized) Hardy-Littlewood inequality to have
$$
\int_{\mathbb R^n}|f(x)g(x)h(x)|dx\leq
\int_0^{\infty}f^*(t)g^*(t)h^*(t)dt=\int_{\mathbb
R^n}f^{\sharp}(x)g^{\sharp}(x)h^{\sharp}(x)dx
$$
Since $q>p$ the following hold
\begin{eqnarray*}
\left(|u(|x|^{\f
pq})|^{\f{q-p}{p}}\right)^{\sharp}(|y|)&=&|u|^{\f{q-p}{p}}(|y|^{\f
pq})\\
\left(|x|^{\f{p-q}{q}}\right)^{\sharp}(|y|)&=&|y|^{\f{p-q}{q}}\\
|\nabla u(|x^{\f pq}|)|^{\sharp}(|y|)&=&|\nabla
u|^{\sharp}(|y|^{\f pq})\ .
\end{eqnarray*}
Thus
\begin{eqnarray*}
\|\nabla v\|_{p}&\leq &\left\{\int_{\mathbb R^n}\left[
|u|^{\f{q-p}{p}}(|x|^{\f pq})|\nabla
u|^{\sharp}(|x|^{\f pq})|x|^{\f{p-q}{q}}\right]^{p}dx\right\}^{1/p}\\
&=&\left\{\f{q}{p}\int_{\mathbb R^n}\left[
|u|^{\f{q-p}{p}}(|y|)|\nabla
u|^{\sharp}(|y|)|y|^{\f{p-q}{p}}\right]^{p}|y|^{n\f{q-p}{p}}dy\right\}^{1/p}\\
&=&\left(\f qp\right)^{\f 1p}\left\{\int_{\mathbb
R^n}|u|^{q-p}(|y|)(|\nabla u
|^{\sharp})^{p}(|y|)||y|^{\f{(n-p)(q-p)}{p}}dy\right\}^{1/p}
\end{eqnarray*}
Let us now multiply and divide by $|y|^{n\f{q-p}{q}}$ to get
\begin{eqnarray}\label{nablavnorm2}
\nonumber \|\nabla v\|_{p}&\leq& \left(\f qp\right)^{\f
1p}\left\{\int_{\mathbb
R^n}\left[|u|^{q-p}|y|^{\f{(n-p)(q-p)}{p}}\cdot |y|^{-n\f{q-p}{q}}
\right]\left[(|\nabla
u|^{\sharp})^{p}(|y|)|y|^{n\f{q-p}{q}}\right]dy\right\}^{1/p}\\
\nonumber &=&\left(\f qp\right)^{\f 1p}\left\{\int_{\mathbb
R^n}\left[|u|^{q-p}\cdot
|y|^{\f{(q-p)[q(n-p)-np]}{qp}}\right]\left[(|\nabla
u|^{\sharp})^{p}(|y|)|y|^{n\f{q-p}{q}}\right]dy\right\}^{1/p}\\
\nonumber  &\leq& \left(\f qp\right)^{\f 1p} \||u|^{q-p}\cdot
|y|^{\f{(q-p)[q(n-p)-np]}{qp}}\|^{1/p}_{\f{q}{q-p}} \|(|\nabla
u|^{\sharp})^{p}(|y|)|y|^{n\f{q-p}{q}}\|^{1/p}_{\f qp}\\
\nonumber & =& \left(\f qp\right)^{\f 1p} \left\{\int_{\mathbb
R^n} |u|^q |y|^{\f{q(n-p)-np}{p}}
dy\right\}^{\f{q-p}{qp}}\left\{\int_{\mathbb R^n} (|\nabla
u|^{\sharp})^q |y|^{n\f{q-p}{p}} dy\right\}^{\f{1}{q}}\\
\nonumber & =& \left(\f qp\right)^{\f
1p}\left\{\omn^{-\f{q(n-p)-np}{np}}\int_0^{\infty} (u^*)^q
t^{\f{q(n-p)-np}{np}}
dt\right\}^{\f{q-p}{qp}}\cdot \\
\nonumber &&\qquad \qquad \cdot \left\{\omn^{-\f{q-p}{p}}\int_0^{\infty}
(|\nabla u|^{*})^q t^{\f{q-p}{p}} dt\right\}^{\f{1}{q}} \\
&=&\left(\f qp\right)^{\f 1p}
\omn^{-\f{(n-p)(q-p)}{np^2}}\|u\|_{p^*,q}^{\f{q-p}{p}}\cdot\|\nabla
u\|_{p,q}
\end{eqnarray}
Now combine Alvino's inequality
$$
\|v\|_{p^*, p}\leq \f{p}{n-p}\omn^{-1/n}\|\nabla v\|_{p}
$$
with \eqref{vnorm2} and \eqref{nablavnorm2} to obtain
\begin{eqnarray*}
  \|u\|_{p^*,q}&=&\left(\f pq\right)^{\f 1q}\omn^{\f{(q-p)(n-p)}{nqp}}\|v\|_{p^*, p}^{\f{p}{q}}\\
  &\leq &\left(\f pq\right)^{\f 1q}\omn^{\f{(q-p)(n-p)}{nqp}}\left(\f{p}{n-p}\right)^{\f pq}\omn^{-\f p{nq}}\|\nabla v\|_{p}^{\f pq}\\
  &\leq & \left(\f{p}{n-p}\right)^{\f pq}\omn^{-\f p{nq}}%
  \|u\|_{p^*,q}^{\f{q-p}{q}}\cdot\|\nabla u\|_{p,q}^{\f
pq}
\end{eqnarray*}
and thus our claim
\begin{eqnarray*}
  \|u\|_{p^*,q}&\leq &
\f{p}{n-p}\,\omn^{-\f 1n} \|\nabla
  u\|_{p,q}
\end{eqnarray*}

\noindent \textbf{Step 2.} The case $q=\infty$.

\noindent Let $ u\in \mathcal{D}_W^{1}L^{p,\infty}(\mathbb R^n)$. Let us
define the auxiliary  function
$$
v(r)=r^{n/p^*}u^{\sharp}(r)
$$
Then
\begin{equation}\label{vinfty}
\|u\|_{p^*, \infty}=\sup_{t>0}
t^{1/p^*}u^*(t)=\omn^{1/p^*}\sup_{r>0}r^{n/p^*}u^\sharp(r)=\omn^{1/p^*}\|v\|_{\infty}
\end{equation}
Since $v$ has finite $L^\infty$  norm, it coincides with the limit
of the $L^\gamma$ norm of $v$, as $\gamma \to +\infty$. For $1<\tilde p <n$
by applying inequality \eqref{SobLorineq-R^n} we have
\begin{align*}
  \|v\|_\gamma^\gamma&=\int_0^{+\infty}(u^\sharp)^\gamma(r)
  r^{\frac{n-p}{p}\gamma}dr=\int_0^{+\infty}\left[u^\sharp(r)r^{\frac{n-p}{p}+\frac{1}{\gamma}}\right]^{\gamma}\frac{dr}{r}\\
  &=\int_0^{+\infty}\left[u^\sharp(r)r^{n/\tilde{p}^*}\right]^{\gamma}\frac{dr}{r}
  =\left[  n\omn^{\gamma/\tilde{p}^*}\right]^{-1}\|u\|_{\tilde p^*,
  \gamma}^\gamma \quad \left(\hbox{where } \frac{n}{\tilde p}=\frac np + \frac 1{\gamma}\right)\\
  &\leq \left[  n\omn^{\gamma/\tilde{p}^*}\right]^{-1}\left[  \frac{\tilde p}{n-\tilde p}\omn^{-1/n}\right]^\gamma
 \|\nabla u\|_{\tilde p, \gamma}^\gamma \\
&= \left[  n\omn^{\gamma/\tilde{p}^*}\right]^{-1}\left[
\frac{\tilde p}{n-\tilde p}\omn^{-1/n}\right]^\gamma
n\omn^{\gamma/\tilde p}
  \int_0^1\left[|\nabla u|^\sharp r^{n/\tilde p}\right]^\gamma
  \frac{dr}{r}\\
  &=\omn^{-\gamma/\tilde p^* + \gamma/\tilde p}\left[\frac{\tilde p}{n-\tilde
  p}\omn^{-1/n}\right]^\gamma\int_0^1\left[|\nabla u|^\sharp r^{n/
  p}\right]^\gamma dr \quad \left(\hbox{since } \frac{n\gamma}{\tilde
  p}-1=\frac{n \gamma}{p}\right)\\
  &= \omn^{-\gamma/\tilde p^* + \gamma/\tilde p}\left[\frac{\tilde p}{n-\tilde
  p}\omn^{-1/n}\right]^\gamma \||\nabla u|^\sharp
  r^{n/p}\|_\gamma^\gamma
\end{align*}
Combining the last inequality with \eqref{vinfty} we obtain
\begin{multline*}
  \|u\|_{p^*, \infty}=\omn^{1/p^*}\|v\|_{\infty}= \omn^{1/p^*}\cdot\lim_{\gamma\to
  +\infty}\|v\|_\gamma\\
  \leq \omn^{1/p^*}\cdot\lim_{\gamma \to +\infty}\omn^{-1/\tilde p^* + 1/\tilde p-1/n}
 \frac{\tilde p}{n-\tilde p} \||\nabla u|^\sharp   r^{n/p}\|_\gamma\\
=\omn^{  1/ p-1/n}
 \frac{ p}{n- p} \||\nabla u|^\sharp
  r^{n/p}\|_\infty
\end{multline*}
since $\tilde p\to p$ and $\tilde p ^*\to p^*$, as $\gamma \to
\infty$. Recalling that
$$
\||\nabla u|^\sharp
  r^{n/p}\|_\infty= \omn^{-1/p}\|\nabla u\|_{p,\infty}
$$
the claim follows.

\subsection{Best constants}\label{sharpness}

\noindent The proof of Theorem
\ref{thm generAlv} will be complete if we prove that the constant $\frac{p}{n-p}\omn^{-1/n}$
appearing in \eqref{SobLorineq-R^n} is sharp for any $1\leq
p<q\leq  \infty$.

\noindent Notice that for $q=+\infty$, the sharpness will be a consequence of the attainability
of the constant which will be proved in Section \ref{att}, hence we consider only the case $q<+\infty$. For this purpose we have to check that the maximizing sequence introduced in
\cite{Alv1} for the case $1\leq q\leq p$ actually works also in the case $p< q \leq +\infty $.

\noindent Consider the radial decreasing function
$$
v_{\varepsilon}(x):=\left\{%
\begin{array}{ll}
   \displaystyle |x|^{-\frac{n-p}{p}+\varepsilon}, & \hbox{if } |x|<1 \\
\\
1-\left(\frac{n-p}{p}-\varepsilon\right)(|x|-1), & \hbox{if } 1\leq |x| < 1+\frac 1{\frac{n-p}{p}-\varepsilon} \\\end{array}%
\right.
$$
whose gradient is given by
$$
|\nabla v_{\varepsilon}(x)|:=\left\{%
\begin{array}{ll}
   \left(\frac{n-p}{p}-\varepsilon\right)\displaystyle |x|^{-\frac{n}{p}+\varepsilon}, & \hbox{if } |x|<1 \\
\\
\left(\frac{n-p}{p}-\varepsilon\right), & \hbox{if } 1\leq |x| < 1+\frac 1{\frac{n-p}{p}-\varepsilon} \\\end{array}%
\right.
$$
which is a decreasing radial function. One has
\begin{equation*}
  \|\nabla v_{\varepsilon}\|_{p,q}^q=n\omn^{\frac qp}\left(\frac{n-p}{p}-\varepsilon\right)^q\frac 1q %
\bigg[\frac{1}{\varepsilon}+\frac{p}{q}\Big(\frac{1}{\frac{n-p}{p}-\varepsilon}+1\Big)^{\frac
np q}-\frac pq\bigg]
\end{equation*}
and
\begin{multline*}
\|v_{\varepsilon}\|_{p^*,q}^q= n\omn^{\frac
q{p^*}}\frac{1}{\varepsilon q}+ n\omn^{\frac
q{p^*}}\int_1^{1+\frac 1{\frac{n-p}{p}-\varepsilon}}\left[1-\left(\frac{n-p}{p}-\varepsilon\right)(r-1)\right]^q%
r^{\frac{n-p}{p}q}\frac{dr}{r}\\
\leq n\omn^{\frac q{p^*}}\frac{1}{\varepsilon q}+ n\omn^{\frac
q{p^*}}\int_1^{1+\frac
1{\frac{n-p}{p}-\varepsilon}}r^{\frac{n-p}{p}q}\frac{dr}{r}\\
= n\omn^{\frac q{p^*}}\frac{1}{\varepsilon q}+ n\omn^{\frac
q{p^*}}\frac{pq}{n-p}\bigg[\Big(\frac{1}{\frac{n-p}{p}-\varepsilon}+1\Big)^{\frac{n-p}p
q}-1\bigg]
\end{multline*}
so that
\begin{multline*}
\frac{\|\nabla
v_{\varepsilon}\|_{p,q}^q}{\|v_{\varepsilon}\|_{p^*,q}^q}\geq
\omn^{\frac qn}
\frac{\left(\frac{n-p}{p}-\varepsilon\right)^q\frac 1q %
\bigg[\frac{1}{\varepsilon}+\frac{p}{q}\Big(\frac{1}{\frac{n-p}{p}-\varepsilon}+1\Big)^{\frac
np q}-\frac pq\bigg]}{\frac{1}{\varepsilon q}+
\frac{pq}{n-p}\bigg[\Big(\frac{1}{\frac{n-p}{p}-\varepsilon}+1\Big)^{\frac{n-p}p
q}-1\bigg]}\longrightarrow \omn^{\frac qn}
\Big(\frac{n-p}{p}\Big)^q
\end{multline*}
as $\varepsilon\to 0$, which is our thesis.

\section{Proof of Theorem \ref{thm equivalence}}\label{weaktostrong}
\noindent As byproduct of Theorem \ref{thm
generAlv} we have proved the implication $(H_p)\Rightarrow
(A_{p,\infty})$. We next prove the converse.
\par \smallskip
\noindent Suppose that $(A_{p,\infty})$ holds, namely
\begin{equation}\label{weakineq}
\|v\|_{p^\ast, \infty}\leq \frac{p}{n-p}\omn^{-\frac 1n}\|\nabla
v\|_{p, \infty}, \qquad v \in \mathcal
D_W^1L^{p,\infty}(\mathbb R^n)
\end{equation}
Then we want to prove Hardy's inequality $(H_p)$ for any
function $u \in \mathcal D^{1,p}$. Actually, thanks to the P\'olya-Szeg\"o
and Hardy-Littlewood inequalities, we may restrict ourselves to
prove the validity of $(H_p)$ for any  $u\in \mathcal D^{1,p,
\sharp}$, that is by density, the class of radially decreasing
Lipschitz function with compact support, such that $\|\nabla u\|_p<+\infty$. By Proposition
\ref{dilation_inv}, we may also assume that $u$ has support in $B_1$, the unit ball centered at the origin.

\noindent Let us define an auxiliary radial function $v$ as follows
$$
v(r)=\int_{r}^{1}\rho^{-\frac
np}\int_{\rho}^{1}|u'(t)|^pt^{n-1}dt\,d\rho
$$
Then $v\in \mathcal C^1(B_1\setminus \{0\})$ and
$$
v'(r)= -\rho^{-\frac np}\int_{r}^1|u'(t)|^pt^{n-1}dt,\qquad
v(1)=0, \,\, \lim_{r\to 0^+}v(r)=+\infty
$$
so that $v$ is radially decreasing and also $|v'|=|\nabla v|$ is
radially decreasing.

\noindent Hence
\begin{multline} \label{weaknorm_nablav}
\|\nabla v\|_{p,\infty}=\omn^{\frac 1p}\sup_{B_1}|\nabla
v|^{\sharp}|x|^{\frac np}= \omn^{\frac 1p}\sup_{0<r<1}|v'|r^{\frac
np}
\\= \omn^{\frac 1p}\int_0^1|u'|^{p}d\rho= \omn^{\frac
1p-1}\|\nabla u\|_p^p
\end{multline}
Moreover, $v\in L^{p^*, \infty}$ since
\begin{multline*}
 \|v\|_{p^*, \infty}=\omn^{\frac
1{p^*}}\sup_{0<r<1}|v|r^{\frac{n-p}p}= \omn^{\frac
1{p^*}}\sup_{0<r<1}r^{\frac{n-p}p} \int_{r}^{1}\rho^{-\frac
np}\int_{\rho}^{1}|u'(t)|^pt^{n-1}dt\,d\rho\\
= \omn^{\frac 1{p^*}}\sup_{0<r<1}r^{\frac{n-p}p}
\int_{r}^{1}|u'(t)|^pt^{n-1}dt\int_{r}^{t}\rho^{-\frac
np}\,d\rho\\=\frac{p}{n-p} \omn^{\frac 1{p^*}}
\sup_{0<r<1}r^{\frac{n-p}p}
\int_{r}^{1}|u'(t)|^pt^{n-1}\left(r^{-\frac{n-p}{p}}-t^{-\frac{n-p}{p}}\right)dt\\
\leq \frac{p}{n-p} \omn^{\frac 1{p^*}} \sup_{0<r<1}
\int_{r}^{1}|u'(t)|^pt^{n-1}dt = \frac{p}{n-p} \omn^{\frac
1{p^*}-1} \|\nabla u\|_p^p<\infty
\end{multline*}
thus we have proved $v\in \mathcal D^{1}_WL^{p, \infty}(\mathbb{R}^n)$.

\noindent Now the idea is to estimate from below the norm $\|v\|_{p^*, \infty}$
with the left hand side of Hardy's inequality
 involving the function $u$. Since
$$
\|v\|_{p^*, \infty}=\omn^{\frac
1{p^*}}\sup_{0<r<1}|v|r^{\frac{n-p}p}
$$
we have to estimate the quantity
$$
 |v|r^{\frac{n-p}p}=r^{\frac{n-p}p}\int_{r}^1\rho^{-\frac
np}\int_{\rho}^1|u'(t)|^pt^{n-1}dt\,d\rho
$$
from below with Hardy's integral involving $u$. Since $-u'(t)=|u'(t)|$, we have
\begin{multline}\label{u^p}
  u^p(\rho)=\left[\int_{\rho}^1|u'|dt\right]^p
  = \left[\int_{\rho}^1|u'|t^{\frac{p-1}{p}\frac np}t^{-\frac{p-1}{p}\frac
np}dt\right]^p\\
\leq
\int_{\rho}^1|u'|^pt^{\frac{p-1}{p}n}dt\left[\int_{\rho}^1t^{-\frac
np}\right]^{p-1}\\\leq
\left(\frac{p}{n-p}\right)^{p-1}\rho^{-\frac{n-p}{p}\,(p-1)}\int_{\rho}^1|u'|^pt^{\frac{p-1}{p}n}dt\ .
\end{multline}
A key step in the proof is now the evaluation of the following
limit, obtained applying de l'H\^opital's theorem (note that we have an indefinite form ${\infty}/{\infty}$, since
$n>p$):
\begin{eqnarray*}
&&\quad\lim_{r\to 0^+}r^{\frac{n-p}{p}}\int_r^1 z^{-\frac np} \int_z^1
u^p(\rho)\rho^{n-p-1}d\rho dz\\
&&\\
&& = \lim_{r\to 0^+}\frac{\int_r^1
z^{-\frac np} \int_z^1 u^p(\rho) \rho^{n-p-1}d\rho dz}{r^{-\frac{n-p}{p}}}\\
&&\\
&&= \frac{p}{n-p} \lim_{r\to 0^+}\frac{ r^{-\frac np} \int_z^1
u^p(\rho)\rho^{n-p-1}d\rho dz}{r^{-\frac{n}{p}}}\\
&&= \frac{p}{n-p}
\int_0^1 u^p(\rho)\rho^{n-p-1}d\rho
\end{eqnarray*}
and thanks to \eqref{u^p}, we have
\begin{multline*}
  \int_0^1 u^p(\rho)\rho^{n-p-1}d\rho=\frac{n-p}{p} \lim_{r\to 0^+}r^{\frac{n-p}{p}}\int_r^1 z^{-\frac np} \int_z^1
u^p(\rho)\rho^{n-p-1}d\rho dz\\
\leq \frac{n-p}{p} \sup_{0<r<1}r^{\frac{n-p}{p}}\int_r^1 z^{-\frac
np} \int_z^1
u^p(\rho)\rho^{n-p-1}d\rho dz \\
\leq  \left(\frac{p}{n-p}\right)^{p-2}\sup_{0<r<1}
r^{\frac{n-p}{p}}\int_r^1 z^{-\frac np} \int_z^1
\rho^{-\frac{n-p}{p}\,(p-1)}\rho^{n-p-1}\int_{\rho}^1|u'|^pt^{\frac{p-1}{p}n}dtd\rho dz\\
= \left(\frac{p}{n-p}\right)^{p-2}\sup_{0<r<1}
r^{\frac{n-p}{p}}\int_r^1 z^{-\frac np} \int_z^1
\rho^{\frac{n}{p}-2}\int_{\rho}^1|u'|^pt^{\frac{p-1}{p}n}dtd\rho dz\ .
\end{multline*}
By Fubini's theorem, we reverse the order of integration in the last integral to get
\begin{multline*}
\left(\frac{p}{n-p}\right)^{p-2}\sup_{0<r<1}
r^{\frac{n-p}{p}}\int_r^1 z^{-\frac np} \int_z^1
|u'|^pt^{\frac{p-1}{p}n}
\int_z^t\rho^{\frac{n}{p}-2}d\rho dt dz\\
\leq\left(\frac{p}{n-p}\right)^{p-1}\sup_{0<r<1}
r^{\frac{n-p}{p}}\int_r^1 z^{-\frac np} \int_z^1
|u'|^pt^{\frac{p-1}{p}n}t^{\frac{n}{p}-1}\rho dt \\
=\left(\frac{p}{n-p}\right)^{p-1}\sup_{0<r<1} r^{\frac{n-p}{p}}v(r)
\end{multline*}
where we have used the fact $\frac{p-1}{p}n+\frac{n}{p}-1=n-1$.

\noindent We conclude the proof by applying the embedding inequality \eqref{weakineq},
\begin{eqnarray*}
\int_{\mathbb R^n}\frac{u^p}{|x|^p}dx&=&\omn \int_0^1 \int_0^1
u^p(\rho)\rho^{n-p-1}d\rho\\
&\leq&
\omn\left(\frac{p}{n-p}\right)^{p-1}\sup_{0<r<1}
r^{\frac{n-p}{p}}v(r)\\
&=& \omn\left(\frac{p}{n-p}\right)^{p-1}\omn^{-\frac 1{p^*}
}\|v\|_{p^*, \infty}\\
&\leq& \omn^{1-\frac 1{p^*}}\left(\frac{p}{n-p}\right)^p\omn^{-\frac
1n}\|\nabla v\|_{p, \infty} \\
&=& \omn^{1-\frac
1p}\left(\frac{p}{n-p}\right)^p\|\nabla v\|_{p, \infty}\\
&=&\left(\frac{p}{n-p}\right)^p\|\nabla u\|_p^p
\end{eqnarray*}
thus Hardy's inequality.

\section{Proof of Theorem \ref{thm attainedornot}}\label{att}

\noindent Here we discuss the attainability of the
sharp embedding constant in \eqref{SobLorineq-R^n}. Observe that for $q=+\infty$, the best embedding
constant in $(A_{p, \infty})$ is attained by the
 function $\psi=|x|^{-\frac{n-p}{p}}$, which is radially decreasing
together with the gradient $|\nabla \psi|=\frac{n-p}{p}|x|^{n/p}$. Hence
 $$
 \|\psi\|_{p^*,
 \infty}=\omn^{1/p^*}\sup_{r>0}r^{-\frac{n-p}{p}}\cdot r^{\frac
 n{p^*}}= \omn^{1/p^*}
 $$
 and
$$
 \|\nabla \psi\|_{p, \infty}=\omn^{1/p}\frac{n-p}{p}\sup_{r>0}r^{-\frac{n}{p}}\cdot
 r^{\frac  n{p}}= \omn^{1/p}\frac{n-p}{p}\ .
 $$
Note that actually we have a whole family of extremal functions, due to the
invariance by dilation proved in Proposition \ref{dilation_inv}.
Moreover, there are plenty of
maximizers for  $(A_{p, \infty})$ in $\mathcal D^1_WL^{p^*,q}(\mathbb{R}^n)$,
since it is enough to have a local asymptotic behavior as $\psi$.

\noindent We next consider the case $p<q<+\infty$. We
will argue by contradiction, proving that the sharp embedding constant is
never attained. Let us suppose the
inequality \eqref{SobLorineq-R^n} is attained at some $q<\infty$.  Following the lines of
Section \ref{H_implies_A}, we have at least one radially decreasing maximizer  $u\in \mathcal
D^1L^{p,q}(\mathbb{R}^n)$, namely a function $u$ such that
\begin{equation}\label{extremal}
\|u\|_{p^*,q}=\frac{p}{n-p}\omn^{-\frac 1n}\|\nabla u\|_{p,q}\ .
\end{equation}
Next define
$$
v(x):=\left[u(|x|^{\f pq})\right]^{\f
qp}=\left[u^*(|x|^{\f{np}q}\omn)\right]^{\f qp}
$$
so that
$$
v^{\sharp}(|x|)=\left[u(|x|^{\f pq})\right]^{\f qp}, \quad
v^*(t)=u^{\sharp}\left((\textstyle{\frac{t}{\omn}})^{1/n}\right)%
=\left[u\left((\textstyle{\frac{t}{\omn}})^{\f{p}{qn}}\right)\right]^{\f
qp}=\left[v^*(t^{\f pq}\omn^{\f{q-p}{q}})\right]^{\f qp}\
$$
By \eqref{vnorm2}, one has $v\in L^{p^*,p}(\mathbb{R}^n)$ with
$$
   \|v\|_{p^*,p}= \left(\f qp\right)^{\f 1p}\omn^{\f{p-q}{p p^*}}\|u\|^{\f qp}_{p^*,q}%
$$
and $\nabla v\in L^p$, with
$$
 \|\nabla v\|_{p}\leq \left(\f qp\right)^{\f 1p}
\omn^{-\f{(n-p)(q-p)}{np^2}}\|u\|_{p^*,q}^{\f{q-p}{p}}\cdot\|\nabla
u\|_{p,q}\ .
$$
By \eqref{extremal} we obtain
$$
\|\nabla v\|_{p}\leq \left(\f qp\right)^{\f 1p}
\left(\frac{p}{n-p}\right)^{\frac{q-p}{p}}
\omn^{-\f{(q-p)}{p^2}}\|\nabla u\|_{p,q}^{\f{q}{p}}
$$
and in turn
\begin{multline*}
 \|v\|_{p^*,p}= \left(\f qp\right)^{\f 1p}\omn^{\f{p-q}{p p^*}}\|u\|^{\f
 qp}_{p^*,q}=\left(\f qp\right)^{\f 1p}\left(\frac{p}{n-p}\right)^{\frac qp}\omn^{-\frac 1n -\frac {q-p}{p^2}}\|\nabla u\|^{\frac qp}_{p,q}\\%
 \geq \left(\f qp\right)^{\f 1p}\left(\frac{p}{n-p}\right)^{\frac qp}\omn^{-\frac 1n -\frac {q-p}{p^2}}%
 \left(\f qp\right)^{-\f 1p}
\left(\frac{p}{n-p}\right)^{-\frac{q-p}{p}} \omn^{\f{(q-p)}{p^2}}
\|\nabla v\|_{p}\\
=\frac{p}{n-p}  \omn^{-\frac 1n}\|\nabla v\|_{p}\ .
\end{multline*}
This directly implies
$$
 \|v\|_{p^*,p}= \frac{p}{n-p}  \omn^{-\frac 1n}\|\nabla v\|_{p}
$$
and, since $v=v^\sharp$ one has
$$
\int_{\mathbb R^n} \frac{v^p}{|x|^p}\ dx = \frac{p}{n-p}
\int_{\mathbb R^n} |\nabla v|^p dx
$$
which  contradicts the non-attainability of Hardy's inequality.

\appendix
\section{Reduction to the radial case}
\noindent Here we follow \cite{ALT}. Let $u\in \mathcal D^1L^{p,q}$, $u\neq
0$, smooth and compactly supported. Notice that because of the invariance by dilation, we can also prescribe the measure of the
support. We aim at proving that there exists $v\in
\mathcal{D}^{1,\sharp}L^{p,q}$ such that $\|v\|_{p^*,q}\geq
\|u\|_{p^*,q}$ and $\|\nabla v\|_{p,q}\leq \|\nabla
u\|_{p,q}$. This yields the following maximization problem
$$
  \max\left\{\|v\|_{p^*, q} : v\in W_0^1L^{p,q}(\Omega), |\nabla v| \leq f \in L^{p,q} \,\,\hbox{ a.e. in
}\Omega, f^*=|\nabla u|^*\right\}\geq \|u\|_{p^*,q}
$$
(the last inequality is trivial: set $v\equiv u, f\equiv|\nabla u|$).
It is known that for any $f\geq 0$, $f\in L^{p,q}(\Omega)$ there
exists a maximal nonnegative sub-solution $v\in
W^1L^{p,q}_0(\Omega)$ of the problem
\begin{equation}\label{sub}
|\nabla v| \leq f
\end{equation}
(see \cite{Li} Prop.~7.2,  p.~164 where the statement is proved for $f\in W_0^{1,p}$
but it can be generalized thanks to the monotonicity of the
decreasing rearrangement).

\noindent Consider the maximization
problem
$$
I(u)=\{\sup \|v\|_{p^*,q} :  v\hbox{ enjoys
\eqref{sub}, } f\geq 0, f\in L^{p,q} \hbox{ and } f^\sharp=|\nabla
u|^\sharp\}\geq \|u\|_{p^*,q}
$$
It was proved in \cite{GN} that if $v$ satisfies \eqref{sub}, with
$f\in L^{p,q}$ such that $f^\sharp=|\nabla u|^\sharp$, then
$$
v^*(t)\leq \frac{1}{n\omn^{1/n}}\int_t^{|\Omega|}
s^{1/n}F(s)\frac{ds}s
$$
for some positive
$
F\in
L^{p,q}(0, |\Omega|)$ such that $F(\omn |x|^n)\prec |\nabla u|^\sharp
(|x|)
$
where, we recall
$$
f\prec g \,\, \Longleftrightarrow \,\, \left\{
\begin{array}{ll}
   \displaystyle \int_0^t f^\ast(s)ds \leq \int_0^t g^\ast(s)ds , & \hbox{ for }  t\in [0, |\Omega|] \\
 \\
\displaystyle \int_0^{|\Omega|} f^\ast(s)ds = \int_0^{|\Omega|} g^\ast(s)ds &  \\
\end{array}%
\right.
$$
The relation $f\prec g$ it is known as the Hardy-Littlewood-Polya
relation between $f$ and $g$. In particular one has
$f^{**}(t)\leq g^{**}(t)$ for any $t$ (see \cite{GN,BS, ALT} for
the definition of $\prec$ and its properties). It turns
out that
$$
I(u)\leq J(u)
$$
where $J(u)$ is the following relaxed maximization problem
\begin{multline*}
J(u)=\Big\{\sup \|w\|_{p^*,q} :
w(|x|)=\frac{1}{n\omn^{1/n}}\int_{\omn |x|^n}^{|\Omega|}
F(s)s^{1/n}\frac {ds}s,\\
F\in L^{p,q}(0,|\Omega|): F(\omn |x|^n)\prec |\nabla
u|^\sharp (|x|), F\geq 0\Big\}
\end{multline*}
By direct calculations we have
\begin{equation}\label{redutc1}
\|w\|_{p^*,q}\leq C_{n,q} \| F\|_{p,q}\leq C\|\nabla u\|_{p,q}
\end{equation}
\noindent Consider the following class
$$
K(|\nabla u|^\sharp):=\{F\in L^{p,q}(0, |\Omega|): F\geq 0, F(\omn
|x|^n)\prec |\nabla u|^\sharp(|x|)\}
$$
for which  the following properties are proved in \cite{ALT}:
\begin{itemize}
\item \ $K(|\nabla u|^\sharp)$ is a convex, weakly compact and
closed set
in $L^{p,q}(0, |\Omega|)$;%
\item \  $K(|\nabla u|^\sharp)$ is the weak closure, in $L^{p,q}(0,
|\Omega|)$, of the
set of positive functions $f$ such that $f^\sharp=|\nabla u|^\sharp$;%
\item \ any extreme point of $K(|\nabla u|^\sharp)$ (namely, any $F$
such that do not exist $F_1, F_2 \in K, F_1 \neq F_2$,
for which $F=\frac{F_1+F_2}{2}$) is equi-measurable with $|\nabla
u|$ with $F^*= |\nabla u|^*$.
\end{itemize}
 (actually the result in \cite{ALT} is proved in $L^p(\Omega)$, but it
can be straightforward generalized).

\noindent Thanks to the previous properties,
if $w_j$ is any maximizing sequence for $J(u)$, then $|\nabla
w_j(|x|)|=F_j(\omn |x|^n)$ is uniformly bounded in
$L^{p,q}(\Omega)$ (or, equivalently, $F_j(s)$ is uniformly bounded
in $L^{p,q}(0,|\Omega|)$) so that $F_j(s)$  converges weakly in
$L^{p,q}(0, |\Omega|)$ to some $F_0(s)\in K(|\nabla u|^\sharp)$,
and $F_j(\omn |x|^n)$  converges weakly in $L^{p,q}(\Omega)$ to
$F_0(\omn |x|^n)$. As a consequence, up to a subsequence,
$\{w_j\}$ converges pointwise if $x\neq 0$, and also weakly into
$L^{p^*,q}(\Omega)$, to the associated function $w_0\in L^{p^*,q}(\Omega)$:
$$
w_j(|x|)\stackrel{L^{p^*,q}(\Omega)}{\rightharpoonup}
w_0(|x|)=\frac{1}{n\omn^{1/n}}\int_{\omn |x|^n}^{|\Omega|}
F_0(s)s^{1/n}\frac {ds}s
$$
%Note that $\|\cdot\|_{p^*,q}$ is not a norm, but only a quasi-norm
%(equivalent to a true norm) if $q>p$, so that, in general, it is
%not weakly lower semicontinuous (we have $\|w_0\|_{p^*,q}\leq
%C\cdot \liminf_j \|w_j\|_{p^*,q} , C>1$).
Once we prove that $w_j\to w_0$ in $L^{p^*,q}(\Omega)$, then $w_0$ will
be  a maximum point for $J(u)$, and hence an extreme point of
$K(|\nabla u|^{\sharp})$. As a consequence, $|\nabla
w_0|^{\sharp }(|x|) \equiv F^*_0 (\omn |x|^n)= |\nabla u|^\sharp
(|x|)$ and thus $|\nabla w_0|^*(s)\equiv F^*_0(s)= |\nabla
u|^* (s)$, and our claim follows:
$$
w_0\in \mathcal D^{1,\sharp}L^{p,q}: \quad
\|w_0\|_{p^*,q}=J(u)\geq I(u)\geq \|u\|_{p^*,q}, \quad \|\nabla
w_0\|_{p,q}=\|\nabla u\|_{p,q}
$$
\noindent In order to prove the strong convergence of $w_j$ in $L^{p^*,q}(\Omega)$,
let us focus on its gradient $F_j(\omn |x|^n)$ whose Lorentz quasi-norm is given by
$$
\|\nabla w_j\|_{L^{p,q}(\Omega)}^q=\|F_j\|_{L^{p,q}(0,
|\Omega|)}^q=\int_0^{|\Omega|}\left[ F_j^{*}t^{\frac
1p}\right]^q\frac{dt}{t}
$$
\noindent Let us recall the maximal function defined for a measurable function $f$ as follows
$$f^{**}(t):=\frac{1}{t}\int_0^tf^*(s)\, ds$$
which defines an equivalent norm in $L^{p,q}(\Omega)$, as long as $p>1$, by
$$|||f|||_{p,q}=\|t^{\frac{1}{p}-\frac{1}{q}}f^{**}(t)\|_{L^q(0,|\Omega|)}$$
Since $F_j(\omn |x|^n)\prec |\nabla u|^\sharp (|x|)$, we have
$$
F_j^{*}(t)\leq F_j^{**}(t)\leq |\nabla u|^{**}(t), \qquad
(F^*_j)^{q}(t)=(F^q_j)^{*}(t)\leq  (F_j^q)^{**}(t)\leq (|\nabla
u|^q)^{**}(t),
$$
and hence
$$
(F_j^q)^{*}(t)t^{\frac qp-1}\leq (|\nabla u|^q)^{**}(t) t^{\frac
qp-1} \in L^{1}(0, |\Omega|),  \quad \hbox{since } |\nabla u|\in
L^{p,q }
$$
On the other hand, since
$F_j(s)$ converges weakly to $F_0(s)$ in  $L^{p,q}(0, |\Omega|)$,
we get
$$
\int_0^{|\Omega|}F_j^*(t)dt= \int_0^{|\Omega|}F_j^*(t)\cdot 1
dt\longrightarrow \int_0^{|\Omega|}F_0^*(t)dt $$
so that, up to a subsequence, $F_j^*(t)$ converges in measure and
a.e. to $F_0^*(t)$. We can then apply the Lebesgue dominated
convergence theorem to the sequence $(F_j^q)^{*}(t)t^{\frac
qp-1}$, obtaining
$$
\|F_j\|_{L^{p,q}(0, |\Omega|)}\longrightarrow \|F_0\|_{L^{p,q}(0,
|\Omega|)}
$$
and eventually $|\nabla w_j|=F_j(\omn |x|^n)\to |\nabla w_0|=
F_0(\omn |x|^n)$ strongly in $L^{p,q}(\Omega)$. From the embedding
$W_0^1L^{p,q}\hookrightarrow L^{p^*,q}$, we
have $w_j\to w_0$ strongly in $L^{p^*,q}(\Omega)$.
\par 

\bigskip

\begin{thebibliography}{11}

\bibitem{Alv1}
A.~Alvino, {\emph{Sulla diseguaglianza di Sobolev in spazi di
Lorentz}} Boll. Un. Mat. Ital. A \textbf{14} (1977), 148--156.

%\bibitem{AFT}
%A.~Alvino, V.~Ferone and G.~Trombetti, \textit{Moser-type inequalities in Lorentz spaces}, Potential Anal. 5 (1996), 273--299.

\bibitem{ALT}
A.~Alvino, P.-L.~Lions and G.~Trombetti {\emph{On optimization
problems with prescribed rearrangements}} Nonlinear Anal.
\textbf{13} (1989) 185--220.

\bibitem{BS} C.~Bennett and R.~Sharpley, \emph{Interpolation of operators},
Pure and Applied
  Mathematics, vol. 129, Boston Academic Press, Inc., 1988.

%\bibitem{BCG} E.~Berchio, D.~Cassani and F.~Gazzola \textit{Hardy-Rellich inequalities with boundary remainder terms and applications}, Manuscripta Mathematica \textbf{131} (2010), 427--458.

\bibitem{BM}
H.~Brezis and M.~Marcus, \textit{Hardy's inequalities revisited. Dedicated to Ennio De Giorgi}, Ann. Scuola Norm. Sup. Pisa \textbf{25} (1997), 217--237 (1998).

\bibitem{BMS}
H.~Brezis, M.~Marcus and I.~Shafrir, \textit{Extremal functions for Hardy's inequality with weight}, J. Funct. Anal. \textbf{171} (2000), 177--191.

\bibitem{BV}
H.~Brezis and J.L.~V\'azquez, \textit{Blow-up solutions of some nonlinear elliptic problems},
Rev. Mat. Univ. Complut. Madrid \textbf{10} (1997), 443--469.

\bibitem{CRT} D.~Cassani, B.~Ruf and C.~Tarsi, \textit{Optimal Sobolev-type inequalities in Lorentz spaces}, Potential Analysis \textbf{39} (2013), 265--285.

\bibitem{CF}
A.~Cianchi and A.~Ferone, \textit{Hardy inequalities with
non-standard remainder terms}, nn. Inst. H. Poincare. Non Lineaire \textbf{25} (2008),
889--906.

\bibitem{CP1}  A.~Cianchi,L.~Pick, \emph{{Sobolev embeddings into BMO, VMO, and $L_{\infty}$}}
 Ark. Mat. {\bf 36} (1998), 317--340.


%\bibitem{CP2}  A.~Cianchi,L.~Pick, \emph{{Sobolev embeddings into spaces of Campanato, Morrey, and H�lder type}}
%J. Math. Anal. Appl. {\bf 282} (2003), 128--150.

\bibitem{C}
S.~Costea, \emph{{Sobolev-Lorentz spaces in the Euclidean setting
and counterexamples}}, Nonlinear Anal. {\bf 152} (2017), 149--182.


\bibitem{D}
E.B.~Davies, \textit{A review of Hardy inequalities}, Oper. Theory Adv. Appl. \textbf{110} (1998), 55--67.

\bibitem{pincho1} B.~Devyver, M.~Fraas and Y.~Pinchover, \textit{Optimal Hardy Weight for Second-Order Elliptic Operator: an answer to a problem of Agmon}, J. Funct. Anal. \textbf{266} (2014), 4422--4489.

\bibitem{pincho2} B.~Devyver and Y.~Pinchover, \textit{Optimal Lp Hardy-type inequalities}, Ann. Inst. H. Poincare. Anal. Non Lineaire \textbf{33} (2016), 93--118.

\bibitem{EKP} D.E.~Edmunds, R.~Kerman and L.~Pick, \textit{{Optimal Sobolev Imbeddings Involving Rearrangement--Invariant
Quasinorms}}, J. Funct. Anal. \textbf{170} (2000), 307--355.

\bibitem{FT}
S.~Filippas and A.~Tertikas, \textit{Optimizing improved Hardy inequalities}, J. Funct. Anal. \textbf{192} (2002), 186--233.

\bibitem{GhoussoubMoradifam} N.~Ghoussoub and A.~Moradifam, Functional inequalities: new perspectives and new applications. Mathematical Surveys and Monographs, 187. American Mathematical Society, Providence, RI, 2013.

\bibitem{GM}
N.~Ghoussoub and A.~Moradifam, \textit{Bessel pairs and optimal Hardy and Hardy-Rellich inequalities}, Math. Ann. \textbf{349} (2011), 1--57.

\bibitem{GN}
E.~Giarusso and D.~Nunziante {\emph{Symmetrization in a class of
first Hamilton-Jacobi equations}}, Nonlinear Anal. \textbf{8}
(1984), 289-–299.

\bibitem{KMP}
A.~Kufner, L.~Maligranda and L.-E.~Persson, \textit{The prehistory of the Hardy inequality}, Amer. Math. Monthly \textbf{113} (2006), 715--732.

\bibitem{L}
G.G.~Lorentz, \textit{Some new functional spaces}, Ann. of Math. \textbf{51} (1950), 37--55.

\bibitem{Li}
P.L.~Lions, \textit{Generalized solution of Hamilton-Jacobi
equations}, Pitman, London (1982).

\bibitem{MMP}
M.~Marcus, V.J.~Mizel and Y.~Pinchover, \textit{On the best constant for Hardy's inequality in $\mathbb{R}^n$}, Trans. Amer. Math. Soc. \textbf{350} (1998), 3237--3255.

\bibitem{Ma}
V.~Maz'ya, \textit{Sobolev spaces with applications to elliptic partial differential equations}, Second ed., Grundlehren der Mathematischen Wissenschaften \textbf{342}, Springer, 2011.

\bibitem{M}
E.~Mitidieri, \textit{A simple approach to Hardy inequalities}, Math. Notes \textbf{67} (2000), 479--486.

\bibitem{MP}
E.~Mitidieri and S.~Pohozaev, \textit{A priori estimates and the absence of solutions of nonlinear partial differential equations and inequalities}, Proc. Steklov Inst. Math. \textbf{234} (2001), 1--362.

\bibitem{OK}
B.~Opic and A.~Kufner, \textit{Hardy-type inequalities}, Pitman Research Notes in Mathematics Series \textbf{219}, 1990.

\bibitem{P}
J.~Peetre, \textit{Espaces d�interpolation et th\'eor\`eme de Soboleff}, Ann. Inst. Fourier \textbf{16} (1966), 279--317.

\bibitem{SanoTakahashi}  M.~Sano and F.~Takahashi, \textit{Scale invariance structures of the critical and the subcritical Hardy inequalities and their improvements}, Calc. Var. Partial Differential Equations \textbf{56} (2017), 14 pp.

\bibitem{Ta}
G..~Talenti, \textit{An inequality between $u^*$ and $|{\rm{grad}}
u^*|$}, Int. Series Num. Math. \textbf{103} (1992), 175--182.


\end{thebibliography}
\end{document}